\documentclass[reqno,10pt]{amsart}
\numberwithin{equation}{section}
\usepackage{latexsym}
\usepackage{amsmath}
\usepackage{amssymb}
\usepackage{mathrsfs}
\usepackage{graphicx,colordvi}
\usepackage{upgreek}
\usepackage{ifthen}
\usepackage{color}

\usepackage[T1]{fontenc}
\usepackage[latin1]{inputenc}

\numberwithin{equation}{section}

\newtheorem{defi}{Definition}[section]
\newtheorem{theorem}[defi]{Theorem}
\newtheorem{lemma}[defi]{Lemma}
\newtheorem{corollary}[defi]{Corollary}

\newtheorem{remark}[defi]{Remark}
\newtheorem{remarks}[defi]{Remarks}

\usepackage{latexsym}
\usepackage{amsmath}
\usepackage{amssymb}

\newcommand{\R}{{\mathbb R}}

\renewcommand{\epsilon}{\varepsilon}

\newcommand{\bb}{{\bb}}

\begin{document}

\title[Non-isothermal quasilinear Westervelt equation]{$L_p$-$L_q$-theory for a quasilinear  non-isothermal Westervelt equation}



\author{Mathias Wilke}
\address{Martin-Luther-Universit\"at Halle-Witten\-berg\\
         Institut f\"ur Mathematik \\
         Halle (Saale), Germany}
\email{mathias.wilke@mathematik.uni-halle.de}


\subjclass[2010]{}

 \keywords{}

\begin{abstract}
We investigate a quasilinear system consisting of the Westervelt equation from nonlinear acoustics and Pennes bioheat equation, subject to Dirichlet or Neumann boundary conditions. The concept of maximal regularity of type $L_p$-$L_q$ is applied to prove local and global well-posedness. Moreover, we show by a parameter trick that the solutions regularize instantaneously. Finally, we compute the equilibria of the system and investigate the long-time behaviour of solutions starting close to equilibria.
\end{abstract}

\maketitle
\section{Introduction}
\noindent
\emph{Thermo-acoustic lensing} describes the effect of how the speed of acoustic waves and the pressure of a region are influenced by the temperature of the underlying tissue. A meanwhile well-accepted model which takes care of this effect consists of the Westervelt equation \cite{Wes63}
\begin{equation}\label{eq:West}
u_{tt} -c^2(\theta)\Delta u-b(\theta)\Delta u_t = k(\theta)(u^2)_{tt},
\end{equation}
describing the propagation of sound in fluidic media, coupled with the so-called bioheat equation proposed by Pennes \cite{Pen48}
\begin{equation}\label{eq:bioheat}
\rho_aC_a\theta_t-\kappa_a\Delta\theta+\rho_bC_bW(\theta-\theta_a)=Q(u_t).
\end{equation}
In \eqref{eq:West}, the function $u=u(t,x)$ denotes the acoustic pressure fluctuation from an ambient value at time $t$ and position $x$.  Furthermore, $c(\theta)>0$ denotes the speed of sound, $b(\theta)>0$ the diffusivity of sound and $k(\theta)>0$ the parameter of nonlinearity.

The physical meaning of the parameters in \eqref{eq:bioheat} are as follows: $\rho_a>0$ and $\kappa_a>0$ denote the ambient density and thermal conductivity, respectively. $C_a>0$ is the ambient heat capacity and $\theta_a>0$ stands for the constant ambient temperature, $\rho_b>0$ is the density of blood, $C_b>0$ is the heat capacity of blood and $W$ denotes the perfusion rate (cooling by blood flow).

The nonlinear function $Q$ models the acoustic energy being absorbed by the surrounding tissue and $Q$ is typically of quadratic type, see Remark \ref{rem:Q}.

Considering \eqref{eq:West}-\eqref{eq:bioheat} in a bounded framework, we have to equip these equations with suitable boundary conditions. In this article, we propose either Dirichlet or Neumann boundary conditions on $u$ and $\theta$. Alltogether, we end up with the following system
\begin{equation}
\label{eq:WestPenn}
\begin{aligned}
u_{tt} -c^2(\theta)\Delta u-b(\theta)\Delta u_t &= k(\theta)(u^2)_{tt},&&\text{in }(0,T)\times\Omega,\\
\rho_aC_a\theta_t-\kappa_a\Delta\theta+\rho_bC_bW(\theta-\theta_a)&=Q(u_t),&&\text{in }(0,T)\times\Omega,\\
     \mathcal{B}_ju&= g_j,&&\text{in }(0,T)\times\partial\Omega,\\
     \mathcal{B}_\ell \theta&= h_\ell,&&\text{in }(0,T)\times\partial\Omega,\\
      (u(0),u_t(0)) &= (u_0,u_1) ,&&\text{in }\Omega,\\
      \theta(0)&= \theta_0 ,&&\text{in }\Omega,
\end{aligned}
\end{equation}
where $(j,\ell)\in\{0,1\}\times\{0,1\}$,
\begin{itemize}
\item $\mathcal{B}_0v=v|_{\partial\Omega}$ (Dirichlet boundary conditions),
\item $\mathcal{B}_1v=\partial_\nu v$ (Neumann boundary conditions),
\end{itemize}
and $u_0,u_1,\theta_0$ denote the initial conditions for $u,u_t,\theta$ at $t=0$.

We observe that as long as $b(\theta)>0$, the term $b(\theta)\Delta u_t$ renders \eqref{eq:West} into a strongly damped wave equation which is of parabolic type. Since
$$(u^2)_{tt}=2u_{tt}u+2(u_t)^2,$$
we see that parabolicity is preserved as long as $|u|$ is sufficiently close to zero. It follows that \eqref{eq:WestPenn} represents a quasilinear parabolic system for the variables $(u,u_t,\theta)$. Therefore, it is reasonable to apply $L_p$-$L_q$-theory in order to solve \eqref{eq:WestPenn}.

The Westervelt equation (with constant temperature) has been subject to a variety of articles over the last decades, see e.g.\ \cite{ClKaVe09,Kal10,KaLa09, KaLa11, KalNi21, KalTha18, MeWi11, SiWi17}, which is just a selection.

To the best knowledge of the author, there is only the article \cite{NikSaid21} which provides analytical results for \eqref{eq:WestPenn} in case of homogeneous Dirichlet boundary conditions for both $u$ and $\theta$ and provided that the diffusivity of sound $b$ \emph{does not} depend on $\theta$. The analysis in \cite{NikSaid21} is based on $L_2$-theory and some (higher-order) energy estimates. To this end, the authors have to equip the initial data with more regularity than is actually needed.

Within the present article, we are interested in the existence and uniqueness of strong solutions to \eqref{eq:WestPenn} having maximal regularity of type $L_p$-$L_q$. In particular, we present optimal conditions on the initial data $(u_0,u_1,\theta_0)$ and the boundary data $(g_j,h_\ell)$, thereby improving the assumptions on $(u_0,u_1,\theta_0)$ in \cite{NikSaid21} (for details, see below). Additionally, we investigate the temporal regularity of the solutions to \eqref{eq:WestPenn} as well as their long-time behaviour.

Our article is structured as follows. In Section \ref{sec:Linearization} we consider a suitable
linearization of \eqref{eq:WestPenn} and we prove optimal regularity results of type
$L_p$-$L_q$ for the resulting parabolic problems. Section \ref{sec:ProofofMR} is devoted to the proof of the following main-result concerning well-posedness
of \eqref{eq:WestPenn} under optimal conditions on the data $(u_0, u_1, \theta_0,g_j, h_\ell)$.
\begin{theorem}\label{thm:main}
Let $T\in (0,\infty)$, $\Omega\subset\mathbb{R}^d$ be a bounded domain with boundary $\partial\Omega\in C^2$ and suppose that $c,b,k\in C^1(\mathbb{R})$ with $b(\tau)\ge b_0>0$ for all $\tau\in\mathbb{R}$. Assume furthermore that $p,q,r,s\in (1,\infty)$ such that
$$\frac{d}{q}<2,\quad \frac{2}{r}+\frac{d}{s}<2$$
and
$$Q\in C^1\left(W_p^1((0,T);L_q(\Omega))\cap L_p((0,T);W_q^2(\Omega));L_r((0,T);L_s(\Omega))\right),$$
with $Q(0)=0$. Let $1-j/2-1/2q\neq 1/p$ and $1-\ell/2-1/2s\neq1/r$.

Then there exists $\delta=\delta(T)>0$ such that for all
$$u_0\in W_q^2(\Omega),\quad u_1\in B_{qp}^{2-2/p}(\Omega),\quad \theta_0\in B_{sr}^{2-2/r}(\Omega),$$
$$g_j\in F_{pq}^{2-j/2-1/2q}((0,T);L_q(\partial\Omega))\cap W_p^1((0,T);W_q^{2-j-1/q}(\partial\Omega))=:Y_j(0,T),$$
$$h_\ell\in F_{rs}^{1-\ell/2-1/2s}((0,T);L_s(\partial\Omega))\cap L_r((0,T);W_s^{2-\ell-1/s}(\partial\Omega)),$$
with
\goodbreak
\begin{itemize}
\item $\mathcal{B}_j u_0=g_j(0)$,
\item $\mathcal{B}_j u_1=\partial_tg_j(0)$ if $1-j/2-1/2q>1/p$,
\item $\mathcal{B}_\ell \theta_0=h_\ell(0)$ if $1-\ell/2-1/2s>1/r$,
\end{itemize}
and
$$\|u_0\|_{W_q^2(\Omega)}+\|u_1\|_{B_{qp}^{2-2/p}(\Omega)}+\|g_j\|_{Y_j(0,T)}\le\delta,$$
there exists a unique solution
$$u\in W_p^2((0,T);L_q(\Omega))\cap W_p^1((0,T);W_q^2(\Omega))$$
$$\theta\in W_r^1((0,T);L_s(\Omega))\cap L_r((0,T);W_s^2(\Omega))$$
of \eqref{eq:WestPenn}. Moreover, the solution $(u,\theta)$ is $C^1$ with respect to the data $(g_j,u_0,u_1,h_\ell,\theta_0)$.
\end{theorem}
\begin{remark}\label{rem:Q}
The nonlinear function $Q$ can for instance be modeled by
$$Q(u_t)=C\cdot(u_t)^2$$
or
$$Q(u_t)=\frac{C}{T}\int_0^T (u_t)^2dt$$
for some constant $C>0$, see e.g. \cite{HallCleve99}, \cite{HallCleveHyn01}, \cite{NorPurr16} . In these cases it can be readily checked that $Q(0)=0$ and
$$Q\in C^1\left(W_p^1((0,T);L_q(\Omega))\cap L_p((0,T);W_q^2(\Omega));L_r((0,T);L_s(\Omega))\right).$$
provided that
$$\frac{2}{p}+\frac{d}{q}<2+\frac{1}{r}+\frac{d}{2s}.$$
\end{remark}
For the proof of Theorem \ref{thm:main} we employ the implicit function theorem and the
results on optimal regularity of the linearization from Section \ref{sec:Linearization}.
In order to compare our results in Theorem \ref{thm:main} with \cite[Theorem 4.1]{NikSaid21}, we consider the very special case $d\in\{1,2,3\}$, $p=q=s=2$ and $g_j=h_\ell=0$ in Theorem \ref{thm:main}.
\begin{corollary}
Let $T\in (0,\infty)$, $d\in\{1,2,3\}$, $\Omega\subset\mathbb{R}^d$ be a bounded domain with boundary $\partial\Omega\in C^2$ and suppose that $c,b,k\in C^1(\mathbb{R})$ with $b(\tau)\ge b_0>0$ for all $\tau\in\mathbb{R}$. Assume furthermore that $r\in (1,\infty)$ such that
$$\frac{2}{r}+\frac{d}{2}<2$$
and
$$Q\in C^1\left(W_2^1((0,T);L_2(\Omega))\cap L_2((0,T);W_2^2(\Omega));L_r((0,T);L_2(\Omega))\right),$$
with $Q(0)=0$.  Let $3/4-\ell/2\neq1/r$.

Then there exists $\delta=\delta(T)>0$ such that for all
$$u_0\in W_2^2(\Omega),\quad u_1\in W_{2}^{1}(\Omega),\quad \theta_0\in B_{2r}^{2-2/r}(\Omega),$$
with
\begin{itemize}
\item $\mathcal{B}_j u_0=0$,
\item $\mathcal{B}_j u_1=0$ if $3/4-j/2>1/2$,
\item $\mathcal{B}_\ell \theta_0=0$ if $3/4-\ell/2>1/r$,
\end{itemize}
and
$$\|u_0\|_{W_2^2(\Omega)}+\|u_1\|_{W_{2}^{1}(\Omega)}\le\delta,$$
there exists a unique solution
$$u\in W_2^2((0,T);L_2(\Omega))\cap W_2^1((0,T);W_2^2(\Omega))$$
$$\theta\in W_r^1((0,T);L_2(\Omega))\cap L_r((0,T);W_2^2(\Omega))$$
of \eqref{eq:WestPenn} with $g_j=h_\ell=0$.
\end{corollary}
Let us compare the well-posedness result \cite[Theorem 4.1]{NikSaid21} concerning \eqref{eq:WestPenn} with homogeneous Dirichlet boundary conditions with our result. In \cite{NikSaid21}, the authors assume that
$$u_0\in W_2^3(\Omega),\quad u_1,\theta_0\in W_2^2(\Omega),$$
(plus compatibility conditions on $\partial\Omega$). Since
$$W_2^2(\Omega)=B_{22}^2(\Omega)\hookrightarrow B_{2r}^{2}(\Omega)\hookrightarrow B_{2r}^{2-2/r}(\Omega)$$
for any $r\ge 2$, we were able to reduce the regularity of the initial data $(u_0,u_1,\theta_0)$. Moreover, a crucial assumption in \cite{NikSaid21} is that the mapping $[\tau\mapsto b(\tau)]$ is constant and furthermore, only homogeneous Dirichlet boundary conditions for $u$ and $\theta$ are considered in \cite{NikSaid21}.
In summary, Theorem \ref{thm:main} generalizes \cite[Theorem 4.1]{NikSaid21} considerably.

In Section \ref{sec:HR} we study the regularity of the solution with respect to the temporal
variable $t$. We use a parameter trick which goes back to Angenent \cite{Ang90}, combined
with the implicit function theorem to prove that the solution enjoys higher regularity with respect to $t$ as soon as $t>0$, see Theorem \ref{thm:HR}. This result
reflects the parabolic regularization effect.

Finally, in Section \ref{sec:Equil}, we compute the equilibria of the system \eqref{eq:WestPenn} if $g_j=0$ and $h_\ell=(1-\ell)\theta_a$ and investigate the long-time behaviour of solutions starting close to equilibria. For the case of Dirichlet boundary conditions for $u$, we prove in Theorem \ref{thm:Equil} that the corresponding equilibria are exponentially stable. Since our assumptions on the initial data $(u_0,u_1,\theta_0)$ as well as on the nonlinearities are less restrictive compared to \cite{NikSaid22}, Theorem \ref{thm:Equil} may be understood of a generalization of \cite[Theorems 2.2 \& 2.3]{NikSaid22}.

\subsection{Notations} Let $p,q\in (1,\infty)$, $s\in\mathbb{R}$. Moreover, let $X$ be a Banach space and  $D\subseteq\R^m$, $m\in\mathbb{N}$. We denote by $W_q^k(D;X)$ the $X$-valued Sobolev space of order $k\in\mathbb{N}_0$ and by $B_{qp}^s(D;X)$ and $F_{qp}^s(D;X)$ the $X$-valued Besov spaces and Triebel-Lizorkin spaces, respectively.

We note on the go that $B_{qq}^s(D;X)=F_{qq}^s(D;X)$ for all $s\in\mathbb{R}$ and $q\in (1,\infty)$. Furthermore, $W_q^s(D;X)=B_{qq}^s(D;X)$ are the $X$-valued Sobolev-Slobodeckii spaces (fractional/generalized Sobolev spaces) as long as $s\in\mathbb{R}_+\setminus\mathbb{N}$.

Finally, we remark that for all $s\in\mathbb{R}$ and $q\in (1,\infty)$ it holds that $F_{q2}^s(D;X)=H_q^s(D;X)$ (the $X$-valued Bessel potential spaces) if and only if $X$ is isomorphic to a Hilbert space and $H_q^1(D;X)=W_q^1(D;X)$ if and only if $X$ is a UMD space.

For the precise definitions and further properties of all these function spaces , we refer the reader to the monograph \cite{Ama19}.

\goodbreak

\section{Maximal regularity of a linearization}\label{sec:Linearization}

Let us consider the two linear problems
\begin{equation}
\label{eq:linHeat}
\begin{aligned}
\rho_aC_a\theta_t-\kappa_a\Delta\theta+\rho_bC_bW\theta&=f_1,&&\text{in }(0,T)\times\Omega,\\
    \mathcal{B}_\ell\theta&= h_\ell,&&\text{in }(0,T)\times\partial\Omega,\\
      \theta(0)&= \theta_0 ,&&\text{in }\Omega,
\end{aligned}
\end{equation}
and
\begin{equation}
\label{eq:linWest}
\begin{aligned}
u_{tt} -a_1(t,x)\Delta u_t-a_2(t,x)\Delta u &= f_2,&&\text{in }(0,T)\times\Omega,\\
     \mathcal{B}_ju&= g_j,&&\text{in }(0,T)\times\partial\Omega,\\
      (u(0),u_t(0)) &= (u_0,u_1) ,&&\text{in }\Omega.
\end{aligned}
\end{equation}
Here $\rho_a,C_a,\rho_b,C_b,\kappa_a,W$ are positive parameters, $a_1,a_2,f,g,u_0,u_1,\theta_0$ are given functions and $(j,\ell)\in\{0,1\}\times\{0,1\}$, where
\begin{itemize}
\item $\mathcal{B}_0v=v|_{\partial\Omega}$ (Dirichlet boundary conditions) or
\item $\mathcal{B}_1v=\partial_\nu v$ (Neumann boundary conditions).
\end{itemize}
For the linear problems \eqref{eq:linHeat} and \eqref{eq:linWest} we have the following results.
\begin{lemma}\label{lem:Heat}
Let $r,s\in (1,\infty)$, $\Omega\subset\R^d$ be a bounded $C^2$-domain and let $T\in (0,\infty)$. Suppose that $1-\ell/2-1/2s\neq 1/r$.

Then there exists a unique solution
$$\theta\in W_r^1((0,T);L_s(\Omega))\cap L_r((0,T);W_s^2(\Omega))$$
of \eqref{eq:linHeat} if and only if
\begin{enumerate}
\item $f_1\in L_r((0,T);L_s(\Omega))$;
\item $h_\ell\in F_{rs}^{1-\ell/2-1/2s}((0,T);L_s(\partial\Omega))\cap L_r((0,T);W_s^{2-\ell-1/s}(\partial\Omega))$;
\item $\theta_0\in B_{sr}^{2-2/r}(\Omega)$
\item $\mathcal{B}_\ell \theta_0=h_\ell(0)$ if $1-\ell/2-1/2s>1/r$.
\end{enumerate}
\end{lemma}
\begin{proof}
  The proof follows from \cite[Theorem 2.3]{DHP07}.
\end{proof}

\begin{lemma}\label{lem:West}
Let $p,q\in (1,\infty)$, $\Omega\subset\R^d$ be a bounded $C^2$-domain and let $T\in (0,\infty)$. Suppose furthermore that $a_1,a_2\in C([0,T]\times \overline{\Omega})$ and $a_1(t,x)\ge \alpha>0$ for all $(t,x)\in [0,T]\times\overline{\Omega}$. Assume that $1-j/2-1/2q\neq1/p$.

Then there exists a unique solution
$$u\in W_p^2((0,T);L_q(\Omega))\cap W_p^1((0,T);W_q^2(\Omega))$$
of \eqref{eq:linWest} if and only if
\begin{enumerate}
\item $f_2\in L_p((0,T);L_q(\Omega))$;
\item $g_j\in F_{pq}^{2-j/2-1/2q}((0,T);L_q(\partial\Omega))\cap W_p^1((0,T);W_q^{2-j-1/q}(\partial\Omega))$;
\item $u_0\in W_q^2(\Omega)$, $u_1\in B_{qp}^{2-2/p}(\Omega)$
\item $\mathcal{B}_j u_0=g_j(0)$ for all $p,q\in (1,\infty)$ and
\item $\mathcal{B}_j u_1=\partial_tg_j(0)$ if $1-j/2-1/2q>1/p$.
\end{enumerate}
\end{lemma}
\begin{proof}
We start with the necessity part. If
$$u\in W_p^2((0,T);L_q(\Omega))\cap W_p^1((0,T);W_q^2(\Omega))$$
is a solution of \eqref{eq:linWest}, then clearly $f\in L_p((0,T;L_q(\Omega))$ by the assumptions on $a_j$ and by the first equation in \eqref{eq:linWest}. Furthermore,
$$W_p^2((0,T);L_q(\Omega))\cap W_p^1((0,T);W_q^2(\Omega))\hookrightarrow W_p^1((0,T);W_q^2(\Omega))\hookrightarrow C([0,T];W_q^2(\Omega))$$
by Sobolev embedding, hence $u_0=u(0)\in W_p^2(\Omega)$. Since
$$\partial_t u\in W_p^1((0,T);L_q(\Omega))\cap L_p((0,T);W_q^2(\Omega)),$$
it follows that $u_1=\partial_t u(0)\in B_{qp}^{2-2/p}(\Omega)$, see e.g. \cite[Theorem 3.4.8]{PS16}.

Concerning the boundary data $g_j$, note that $\mathcal{B}_j u\in W_p^1((0,T);W_q^{2-j-1/q}(\partial\Omega))$ and
$$\mathcal{B}_j \partial_tu\in F_{pq}^{1-j/2-1/2q}((0,T);L_q(\partial\Omega))\cap L_p((0,T);W_q^{2-j-1/q}(\partial\Omega)),$$
see e.g. \cite[Chapter VIII]{Ama19}, \cite[Section 6]{DHP07} or \cite[Section 6.2]{PS16}.
Since, by Theorem 1.2 in \cite{MeyrVer12},
$$W_p^1((0,T);W_q^{2-j-1/q}(\partial\Omega))\hookrightarrow F_{pq}^{1-j/2-1/2q}((0,T);L_q(\partial\Omega)),$$
we obtain
$$g_j,\partial_t g_j\in F_{pq}^{1-j/2-1/2q}((0,T);L_q(\partial\Omega))\cap L_p((0,T);W_q^{2-j-1/q}(\partial\Omega)),$$
hence
$$g_j\in F_{pq}^{2-j/2-1/2q}((0,T);L_q(\partial\Omega))\cap W_p^1((0,T);W_q^{2-j-1/q}(\partial\Omega)),$$
by \cite[Proposition 3.10]{MeyrVer12}.

Since $\mathcal{B}_ju=g_j\in W_p^1((0,T);W_q^{2-j-1/q}(\partial\Omega))$ and
$$W_p^1((0,T);W_q^{2-j-1/q}(\partial\Omega))\hookrightarrow C([0,T];W_q^{2-j-1/q}(\partial\Omega)),$$
we necessarily have $\mathcal{B}_ju_0=g_j(0)$ for all $p,q\in (1,\infty)$. Furthermore,
$$\mathcal{B}_j\partial_tu=\partial_tg_j\in F_{pq}^{1-j/2-1/2q}((0,T);L_q(\partial\Omega))\cap L_p((0,T);W_q^{2-j-1/q}(\partial\Omega))$$
hence $\mathcal{B}_j u_1=\partial_t g_j(0)$ provided $1-j/2-1/2q>1/p$ (see e.g. \cite[Corollary VIII.1.1.4]{Ama19}) or \cite[Section 6.2]{PS16}).

We now prove that the conditions in Lemma \ref{lem:West} are also sufficient. To this end, we first consider the problem
\begin{equation}
\label{eq:linWest2}
\begin{aligned}
v_{t} -a_1(t,x)\Delta v &= f_2,&&\text{in }(0,T)\times\Omega,\\
     \mathcal{B}_jv&= \partial_tg_j,&&\text{in }(0,T)\times\partial\Omega,\\
      v(0) &= u_1 ,&&\text{in }\Omega.
\end{aligned}
\end{equation}
By \cite[Theorem 2.3]{DHP07} there exists a unique solution
$$v\in W_p^1((0,T);L_q(\Omega))\cap L_p((0,T);W_q^2(\Omega))$$
of \eqref{eq:linWest2}. Define
$$u(t,x)=u_0(x)+\int_0^t v(s,x)ds,\quad t\in [0,T].$$
Then
$$u\in W_p^2((0,T);L_q(\Omega))\cap W_p^1((0,T);W_q^2(\Omega)),$$
$u(0,x)=u_0(x)$, $\mathcal{B}_j u(t,x)=g_j(t,x)$ (by the compatibility condition on $u_0$) and $\partial_t^k u(t,x)=\partial_t^{k-1}v(t,x)$ for $k\in\{1,2\}$. Consequently, the function $u$ is the unique solution of the problem
\begin{equation}
\label{eq:linWest3}
\begin{aligned}
u_{tt} -a_1(t,x)\Delta u_t &= f_2,&&\text{in }(0,T)\times\Omega,\\
     \mathcal{B}_ju&= g_j,&&\text{in }(0,T)\times\partial\Omega,\\
      (u(0),u_t(0)) &= (u_0,u_1) ,&&\text{in }\Omega.
\end{aligned}
\end{equation}
Uniqueness can be seen as follows. If $u_1$ and $u_2$ are two solutions of \eqref{eq:linWest3}, then $u_1-u_2$ solves \eqref{eq:linWest3} with $(f_2,g_j,u_0,u_1)=0$ and therefore, $\partial_t (u_1-u_2)$ solves \eqref{eq:linWest2} with $(f_2,g_j,u_1)=0$, wherefore $\partial_t (u_1-u_2)=0$. Since $(u_1-u_2)(0)=0$, it follows that $u_1-u_2=0$, hence $u_1=u_2$.

Next, we consider the problem
\begin{equation}
\label{eq:linWest4}
\begin{aligned}
w_{tt} -a_1(t,x)\Delta w_t-a_2(t,x)\Delta w &= \tilde{f}_2,&&\text{in }(0,T)\times\Omega,\\
     \mathcal{B}_jw&= 0&&\text{in }(0,T)\times\partial\Omega,\\
      (w(0),w_t(0)) &= (0,0) ,&&\text{in }\Omega,
\end{aligned}
\end{equation}
for given $\tilde{f}_2\in L_p((0,T);L_q(\Omega))$. Note that for a sufficiently  smooth solution, it holds that $\mathcal{B}_jw_t= 0$ in $(0,T)\times\partial\Omega$. We reformulate \eqref{eq:linWest4} as a first order system. To this end, let $z=(z_1,z_2)=(w,w_t)$ and $F=(0,\tilde{f}_2)$. Then
\begin{equation}\label{eq:1stOrder}
z_t=\begin{pmatrix}
0 & I\\ 0 & a_1(t,x)\Delta
\end{pmatrix}z+\begin{pmatrix}
0 & 0\\ a_2(t,x)\Delta & 0
\end{pmatrix}z+F,
\end{equation}
with the initial condition $z(0)=0$ in $\Omega$ and the boundary condition $\mathcal{B}_jz= 0$ in $(0,T)\times\partial\Omega$. Let 
$$D(\Delta_j)=\{w\in W_q^2(\Omega)\mid \mathcal{B}_j w=0\ \text{on}\  \partial\Omega\}$$
and define
$X_0=D(\Delta_j)\times L_q(\Omega)$ as well as $X_1=D(\Delta_j)\times D(\Delta_j)$. Furthermore, let
$$A_1(t)=\begin{pmatrix}
0 & I\\ 0 & a_1(t,\cdot)\Delta
\end{pmatrix}\quad\text{and}\quad A_2(t)=\begin{pmatrix}
0 & 0\\ a_2(t,\cdot)\Delta & 0
\end{pmatrix}.$$
Then, we have $A_1\in C([0,T];\mathcal{L}(X_1,X_0))$ and $A_2\in C([0,T];\mathcal{L}(X_0,X_0))$. Moreover, $A_1(t)$ has the property of $L_p$-maximal regularity in $X_0$ for any $t\in [0,T]$.

By \cite[Theorem 3.1]{PrSchn01} there exists a unique solution
$$z\in W_p^1((0,T);X_0)\cap L_p((0,T);X_1)$$
of the equation \eqref{eq:1stOrder} subject to the initial  condition $z(0)=0$. This in turn yields the existence and uniqueness of a solution 
$$w\in W_p^2((0,T);L_q(\Omega))\cap W_p^1((0,T);W_q^2(\Omega)),$$
of \eqref{eq:linWest4}. Finally, we solve \eqref{eq:linWest3} to obtain a solution 
$$\tilde{u}\in W_p^2((0,T);L_q(\Omega))\cap W_p^1((0,T);W_q^2(\Omega)).$$ 
Then, we solve \eqref{eq:linWest4} with $\tilde{f}_2:=a_2\Delta\tilde u\in L_p((0,T);L_q(\Omega))$ to obtain a solution 
$$\tilde{w}\in W_p^2((0,T);L_q(\Omega))\cap W_p^1((0,T);W_q^2(\Omega)).$$ 
It is readily checked that the sum 
$$u:=\tilde{u}+\tilde{w}\in W_p^2((0,T);L_q(\Omega))\cap W_p^1((0,T);W_q^2(\Omega))$$ 
is the unique solution of \eqref{eq:linWest}.

\end{proof}

Finally, let us consider the following coupled linear problem
\begin{equation}
\label{eq:linWestHeat}
\begin{aligned}
u_{tt} -a_1(t,x)\Delta u_t-a_2(t,x)\Delta u &= f_2,&&\text{in }(0,T)\times\Omega,\\
\rho_aC_a\theta_t-\kappa_a\Delta\theta+\rho_bC_bW\theta+Bu_t&=f_1,&&\text{in }(0,T)\times\Omega,\\
\mathcal{B}_ju&= g_j,&&\text{in }(0,T)\times\partial\Omega,\\
    \mathcal{B}_\ell\theta&= h_\ell,&&\text{in }(0,T)\times\partial\Omega,\\
      (u(0),u_t(0)) &= (u_0,u_1) ,&&\text{in }\Omega,\\
          \theta(0)&= \theta_0 ,&&\text{in }\Omega.
\end{aligned}
\end{equation}

\begin{lemma}\label{lem:WestHeat}
Let $\Omega\subset\R^d$ be a bounded $C^2$-domain, $T\in (0,\infty)$ and let $p,q,r,s\in (1,\infty)$ such that
$$B:W_p^1((0,T);L_q(\Omega))\cap L_p((0,T);W_q^2(\Omega))\to L_r((0,T);L_s(\Omega))$$
is linear and bounded. Suppose furthermore that $a_1,a_2\in C([0,T]\times \overline{\Omega})$ and $a_1(t,x)\ge \alpha>0$ for all $(t,x)\in [0,T]\times\overline{\Omega}$. Assume that $1-j/2-1/2q\neq1/p$ and $1-\ell/2-1/2s\neq1/r$.

Then there exists a unique solution
$$u\in W_p^2((0,T);L_q(\Omega))\cap W_p^1((0,T);W_q^2(\Omega)),$$
$$\theta\in W_r^1((0,T);L_s(\Omega))\cap L_r((0,T);W_s^2(\Omega))$$
of \eqref{eq:linWestHeat} if and only if
\begin{enumerate}
\item $f_1\in L_r((0,T);L_s(\Omega))$;
\item $f_2\in L_p((0,T);L_q(\Omega))$;
\item $g_j\in F_{pq}^{2-j/2-1/2q}((0,T);L_q(\partial\Omega))\cap W_p^1((0,T);W_q^{2-j-1/q}(\partial\Omega))$;
\item $h_\ell\in F_{rs}^{1-\ell/2-1/2s}((0,T);L_s(\partial\Omega))\cap L_r((0,T);W_s^{2-\ell-1/s}(\partial\Omega))$;
\item $u_0\in W_q^2(\Omega)$, $u_1\in B_{qp}^{2-2/p}(\Omega)$;
\item $\theta_0\in B_{sr}^{2-2/r}(\Omega)$;
\item $\mathcal{B}_j u_0=g_j(0)$ for all $p,q\in (1,\infty)$;
\item $\mathcal{B}_j u_1=\partial_tg_j(0)$ if $1-j/2-1/2q>1/p$;
\item $\mathcal{B}_\ell \theta_0=h_\ell(0)$ if $1-\ell/2-1/2s>1/r$.
\end{enumerate}
\end{lemma}
\begin{proof}
  Necessity of the conditions follows as in the proofs of Lemma \ref{lem:Heat} and \ref{lem:West}.

  To prove sufficiency, one first solves $\eqref{eq:linWestHeat}_{1,3,5}$ for $u$ by Lemma \ref{lem:West}. Then, by the assumption on $B$, it follows that $B u_t\in L_r((0,T);L_s(\Omega))$ is a given function. Therefore, we may solve $\eqref{eq:linWestHeat}_{2,4,6}$ by Lemma \ref{lem:Heat} to obtain $\theta$.
\end{proof}

\section{Proof of Theorem \ref{thm:main}}\label{sec:ProofofMR}

We will prove Theorem \ref{thm:main} by means of the implicit function theorem. To this end, for fixed but arbitrary $T>0$, let us first introduce the function spaces
$$\mathbb{E}_0^u:=L_p((0,T);L_q(\Omega)),\quad \mathbb{E}_0^\theta:=L_r((0,T);L_s(\Omega)),$$
$$\mathbb{E}_1^u:=W_p^2((0,T);L_q(\Omega))\cap W_p^1((0,T);W_q^2(\Omega)),$$
$$\mathbb{E}_1^\theta:=W_r^1((0,T);L_s(\Omega))\cap L_r((0,T);W_s^2(\Omega)),$$
$$Y^u_j:=F_{pq}^{2-j/2-1/2q}((0,T);L_q(\partial\Omega))\cap W_p^1((0,T);W_q^{2-j-1/q}(\partial\Omega)),$$
$$Y^\theta_\ell:=F_{rs}^{1-\ell/2-1/2s}((0,T);L_s(\partial\Omega))\cap L_r((0,T);W_s^{2-\ell-1/s}(\partial\Omega)),$$
$$X_\gamma^u:=W_q^2(\Omega)\times B_{qp}^{2-2/p}(\Omega),\quad X_\gamma^\theta:=B_{sr}^{2-2/r}(\Omega),$$
and
$$\mathbb{Y}_j^u:=\{(\tilde{g}_j,(\tilde{u}_0,\tilde{u}_1))\in Y_j^u\times X_\gamma^u:\mathcal{B}_j \tilde{u}_1=\partial_t\tilde{g}_j(0)\ \text{if}\ 1-j/2-1/2q>1/p,\ \mathcal{B}_j \tilde{u}_0=\tilde{g}_j(0)\},$$
$$\mathbb{Y}_\ell^\theta:=\{(\tilde{h}_\ell,\tilde{\theta}_0)\in Y_\ell^\theta\times X_\gamma^\theta:\mathcal{B}_\ell \tilde{\theta}_0=\tilde{h}_\ell(0)\ \text{if}\ 1-\ell/2-1/2s>1/r\}.$$
Next, we define a function
$$\Phi:\mathbb{E}_1^u\times \mathbb{E}_1^\theta\times \mathbb{Y}_j^u\times \mathbb{Y}_\ell^\theta\to \mathbb{E}_0^u\times \mathbb{E}_0^\theta\times\mathbb{Y}_j^u\times \mathbb{Y}_\ell^\theta ,$$
by
$$\Phi(u,\theta,g_j,u_0,u_1,h_\ell,\theta_0)=
	\begin{pmatrix}
	u_{tt} -c^2(\theta)\Delta u-b(\theta)\Delta u_t - k(\theta)(u^2)_{tt}\\
	\rho_aC_a\theta_t-\kappa_a\Delta\theta+\rho_bC_bW(\theta-\theta_a)-Q(u_t)\\
	\mathcal{B}_ju-g_j\\
	u(0)-u_0\\
	u_t(0)-u_1\\
	\mathcal{B}_\ell\theta-h_\ell\\
	\theta(0)-\theta_0
	\end{pmatrix}.
$$
Note that
$$(u^2)_{tt}=2u_{tt}\cdot u+2(u_t)^2$$
for each $u\in\mathbb{E}_1^u$. Since (by assumption) $d/q<2$, it holds that
$$\mathbb{E}_1^u\hookrightarrow W_p^1((0,T);W_q^2(\Omega))\hookrightarrow C([0,T];W_q^2(\Omega))\hookrightarrow C([0,T];C(\overline{\Omega})),$$
hence
$$\|u_{tt}\cdot u\|_{\mathbb{E}_0^u}\le C\cdot\|u\|_{\mathbb{E}_1^u}^2,$$
for some constant $C>0$.
Let
$$\dot{\mathbb{E}}_1^u:=W_p^1((0,T);L_q(\Omega))\cap L_p((0,T);W_q^2(\Omega)).$$
Then,
$$\dot{\mathbb{E}}_1^u\hookrightarrow L_{2p}((0,T);L_{2q}(\Omega))$$
provided $1/p+d/2q<2$, which is satisfied, since $d/q<2$ and $p>1$. Therefore
$$\|(u_t)^2\|_{\mathbb{E}_0^u}\le C\|u_t\|_{\dot{\mathbb{E}}_1^u}^2\le C\|u\|_{\mathbb{E}_1^u}^2,$$
for some constant $C>0$. Finally, note that
$$\mathbb{E}_1^\theta\hookrightarrow C([0,T];C(\overline{\Omega}))$$
since (by assumption) $2/r+d/s<2$. It follows that
$$\|k(\theta)(u^2)_{tt}\|_{\mathbb{E}_0^u}\le \|k(\theta)\|_{L_\infty((0,T);L_\infty(\Omega))}\|(u^2)_{tt}\|_{\mathbb{E}_0^u}\le C\|k(\theta)\|_{L_\infty((0,T);L_\infty(\Omega))}\|u\|_{\mathbb{E}_1^u}^2,$$
as well as
$$\|b(\theta)\Delta u_t\|_{\mathbb{E}_0^u}\le \|b(\theta)\|_{L_\infty((0,T);L_\infty(\Omega))}\|\Delta u_t\|_{\mathbb{E}_0^u}\le C\|b(\theta)\|_{L_\infty((0,T);L_\infty(\Omega))}\|u\|_{\mathbb{E}_1^u}$$
for some constant $C>0$, since $b,k\in C(\mathbb{R})$. Similarly, we obtain
$$\|c^2(\theta)\Delta u\|_{\mathbb{E}_0^u}\le C\|c^2(\theta)\|_{L_\infty((0,T);L_\infty(\Omega))}\|u\|_{\mathbb{E}_1^u}.$$
In summary, the mapping $\Phi$ is well-defined and
$$\Phi\in C^1(\mathbb{E}_1^u\times \mathbb{E}_1^\theta\times \mathbb{Y}_j^u\times \mathbb{Y}_\ell^\theta;\mathbb{E}_0^u\times \mathbb{E}_0^\theta\times\mathbb{Y}_j^u\times \mathbb{Y}_\ell^\theta),$$
by the assumptions on $b,c,k$ and $Q$.

Let $(h_\ell^*,\theta_0^*)\in\mathbb{Y}_\ell^\theta$ be given and denote by $\theta^*\in \mathbb{E}_1^\theta$ the unique solution of
\begin{equation}
\label{eq:linHeat2}
\begin{aligned}
\rho_aC_a\theta_t^*-\kappa_a\Delta\theta^*+\rho_bC_bW(\theta^*-\theta_a)&=0,&&\text{in }(0,T)\times\Omega,\\
    \mathcal{B}_\ell\theta^*&= h_\ell^*,&&\text{in }(0,T)\times\partial\Omega,\\
      \theta^*(0)&= \theta_0^* ,&&\text{in }\Omega,
\end{aligned}
\end{equation}
which exists thanks to Lemma \ref{lem:Heat} . Then, obviously, $\Phi(0,\theta^*,0,0,0,h_\ell^*,\theta_0^*)=0$ and
$$D_{(u,\theta)}\Phi(0,\theta^*,0,0,0,h_\ell^*,\theta_0^*)(\hat{u},\hat{\theta})=
\begin{pmatrix}
	\hat{u}_{tt} -c^2(\theta^*)\Delta \hat{u}-b(\theta^*)\Delta \hat{u}_t \\
	\rho_aC_a\hat{\theta}_t-\kappa_a\Delta\hat{\theta}+\rho_bC_bW\hat{\theta}-Q'(0)\hat{u}_t\\
	\mathcal{B}_j\hat{u}\\
	\hat{u}(0)\\
	\hat{u}_t(0)\\
	\mathcal{B}_\ell\hat{\theta}\\
	\hat{\theta}(0)
	\end{pmatrix},
$$
where $D_{(u,\theta)}\Phi$ denotes the total derivative of $\Phi$ with respect to $(u,\theta)$.
By Lemma \ref{lem:WestHeat}, the linear operator
$$D_{(u,\theta)}\Phi(0,\theta^*,0,0,0,h_\ell^*,\theta_0^*):\mathbb{E}_1^u\times \mathbb{E}_1^\theta\to \mathbb{E}_0^u\times \mathbb{E}_0^\theta\times\mathbb{Y}_j^u\times \mathbb{Y}_\ell^\theta$$
is invertible. Hence, the implicit function theorem yields some $\delta>0$ and the existence of a $C^1$-function
$$\psi:\mathbb{B}_{\mathbb{Y}_j^u\times \mathbb{Y}_\ell^\theta}((0,0,0,h_\ell^*,\theta_0^*),\delta)\to \mathbb{E}_1^u\times \mathbb{E}_1^\theta$$
such that $(0,\theta^*)= \psi(0,0,0,h_\ell^*,\theta_0^*)$ and
$$\Phi(\psi(g_j,u_0,u_1,h_\ell,\theta_0),(g_j,u_0,u_1,h_\ell,\theta_0))=0$$
for all
$$(g_j,u_0,u_1,h_\ell,\theta_0)\in \mathbb{B}_{\mathbb{Y}_j^u\times \mathbb{Y}_\ell^\theta}((0,0,0,h_\ell^*,\theta_0^*),\delta).$$
This completes the proof of Theorem \ref{thm:main}.
\goodbreak
\begin{remarks}\mbox{}
\begin{enumerate}
\item
It is possible to generalize \eqref{eq:WestPenn} to the case where the nonlinearities $c,b$ or $k$ in \eqref{eq:WestPenn} depend not only on $\theta$ but also on $\nabla\theta$. In this case, the condition
$$\frac{2}{r}+\frac{d}{s}<2$$
in Theorem \ref{thm:main} has to be replaced by the stronger condition
$$\frac{2}{r}+\frac{d}{s}<1,$$
since in this case $B_{sr}^{2-2/r}(\Omega)\hookrightarrow C^1(\overline{\Omega})$.
Then all assertions of Theorem \ref{thm:main} remain valid provided $c,b,k\in C^1(\mathbb{R}\times\mathbb{R}^d)$.

\item The nonlinearity $(u^2)_{tt}$ in \eqref{eq:WestPenn} can be replaced by the more general formulation $(f(u)u_t)_t$, where $f\in C^2(\mathbb{R})$ with $f(0)=0$. This kind of nonlinearity has been derived in \cite{KalRund21}. If $f(s)=2s$, we are in the situation of \eqref{eq:WestPenn}.

\end{enumerate}
\end{remarks}

\goodbreak

\section{Higher Regularity}\label{sec:HR}

We intend to prove that the solution $(u,\theta)$ in Theorem \ref{thm:main} enjoys more time regularity as soon as $t>0$.

Let $(u_*,\theta_*)\in\mathbb{E}_1^u\times\mathbb{E}_1^\theta$ be the unique solution to \eqref{eq:WestPenn} with $g_j=h_\ell=0$ on the interval $[0,T]$ which exists thanks to Theorem \ref{thm:main}. For fixed $\varepsilon\in (0,1)$ and $t\in [0,T/(1+\varepsilon)]$, $\lambda\in (1-\varepsilon,1+\varepsilon)$, we define $u_\lambda(t):=u_*(\lambda t)$ and $\theta_\lambda(t):=\theta_*(\lambda t)$. Then $(u_\lambda,\theta_\lambda)$ is a solution of
\begin{equation}
\label{eq:WestPennHighReg}
\begin{aligned}
\partial_t^2u_{\lambda} -\lambda^2c^2(\theta_\lambda)\Delta u_\lambda-\lambda b(\theta_\lambda)\Delta \partial_tu_\lambda &= k(\theta_\lambda)(u_\lambda^2)_{tt},&&\text{in }(0,T_\varepsilon)\times\Omega,\\
\rho_aC_a\partial_t\theta_\lambda-\lambda\kappa_a\Delta\theta_\lambda+\lambda\rho_bC_bW(\theta_\lambda-\theta_a)&=\lambda Q(\lambda^{-1}\partial_tu_\lambda),&&\text{in }(0,T_\varepsilon)\times\Omega,\\
     \mathcal{B}_ju_\lambda&= 0,&&\text{in }(0,T_\varepsilon)\times\partial\Omega,\\
     \mathcal{B}_\ell \theta_\lambda&= 0,&&\text{in }(0,T_\varepsilon)\times\partial\Omega,\\
      (u_\lambda(0),\partial_t u_\lambda(0)) &= (u_0,\lambda u_1) ,&&\text{in }\Omega,\\
      \theta_\lambda(0)&= \theta_0 ,&&\text{in }\Omega,
\end{aligned}
\end{equation}
where $T_\varepsilon:=T/(1+\varepsilon)$, $(u_0,u_1)\in X_\gamma^u$, $\theta_0\in X_\gamma^\theta$ with
$$\mathcal{B}_j {u}_1=0\ \text{if}\ 1-j/2-1/2q>1/p,\ \mathcal{B}_j {u}_0=0$$
and $\mathcal{B}_\ell {\theta}_0=0$ if $1-\ell/2-1/2s>1/r$. For those fixed initial data, we define a function
$$\Phi:(1-\varepsilon,1+\varepsilon)\times \mathbb{E}_1^u\times\mathbb{E}_1^\theta\to \mathbb{E}_0^u\times \mathbb{E}_0^\theta\times\mathbb{Y}_j^u\times \mathbb{Y}_\ell^\theta$$
by
$$\Phi(\lambda,u,\theta)=
	\begin{pmatrix}
	u_{tt} -\lambda^2c^2(\theta)\Delta u-\lambda b(\theta)\Delta u_t - k(\theta)(u^2)_{tt}\\
	\rho_aC_a\theta_t-\lambda\kappa_a\Delta\theta+\lambda\rho_bC_bW(\theta-\theta_a)-\lambda Q(\lambda^{-1}u_t)\\
	\mathcal{B}_ju\\
	u(0)-u_0\\
	u_t(0)-\lambda u_1\\
	\mathcal{B}_\ell\theta\\
	\theta(0)-\theta_0
	\end{pmatrix}.
$$
Under the conditions of Theorem \ref{thm:main}, the mapping $\Phi$ is $C^1$. Furthermore, we observe $\Phi(1,u_*,\theta_*)=0$ and
$$D_{(u,\theta)}\Phi(1,u_*,\theta_*)(\hat{u},\hat{\theta})=
\begin{pmatrix}
	\hat{u}_{tt} -c^2(\theta_*)\Delta \hat{u}-b(\theta_*)\Delta \hat{u}_t-A_1(u_*,\theta_*)\hat{\theta}-A_2(u_*,\theta_*)\hat{u} \\
	\rho_aC_a\hat{\theta}_t-\kappa_a\Delta\hat{\theta}+\rho_bC_bW\hat{\theta}-Q'((u_*)_t)\hat{u}_t\\
	\mathcal{B}_j\hat{u}\\
	\hat{u}(0)\\
	\hat{u}_t(0)\\
	\mathcal{B}_\ell\hat{\theta}\\
	\hat{\theta}(0)
	\end{pmatrix},
$$
where
$$A_1(u_*,\theta_*)\hat{\theta}:=[2c'(\theta_*)c(\theta_*)\Delta u_*+b'(\theta_*)\Delta (u_*)_t+k'(\theta_*)((u_*)^2)_{tt}]\hat{\theta}$$
and $A_2(u_*,\theta_*)\hat{u}=2k(\theta_*)(u_*\hat{u})_{tt}$.

A Neumann series argument implies that
$$D_{(u,\theta)}\Phi(1,u_*,\theta_*):\mathbb{E}_1^u\times\mathbb{E}_1^\theta\to \mathbb{E}_0^u\times \mathbb{E}_0^\theta\times\mathbb{Y}_j^u\times \mathbb{Y}_\ell^\theta$$
is invertible provided that the norm
$\|u_*\|_{\mathbb{E}_1^u}$ is sufficiently small, which
follows readily by decreasing $\|(u_0,u_1)\|_{X_\gamma^u}$, if necessary. Note that then also $\|\theta_*-\theta^*\|_{\mathbb{E}_1^\theta}$ is small, where $\theta^*$ solves \eqref{eq:linHeat2} with $h_\ell^*=0$ and $\theta_0^*=\theta_0$.

Therefore, by the implicit function theorem, there exists $r\in (0,\varepsilon)$ and a unique mapping $\phi\in C^1((1-r,1+r);\mathbb{E}_1^u\times\mathbb{E}_1^\theta)$ such that $\Phi(\lambda,\phi(\lambda))=0$ for all $\lambda\in (1-r,1+r)$ and $\phi(1)=(u_*,\theta_*)$. By uniqueness, it holds that $(u_\lambda,\theta_\lambda)=\phi(\lambda)$, hence
$$[\lambda\mapsto (u_\lambda,\theta_\lambda)]\in C^1((1-r,1+r);\mathbb{E}_1^u\times\mathbb{E}_1^\theta).$$
Since $\partial_\lambda(u_\lambda(t),\theta_\lambda(t))|_{\lambda=1}=t\partial_t(u_*,\theta_*)$, we obtain
$$[t\mapsto t\partial_t(u_*(t),\theta_*(t))]\in \mathbb{E}_1^u\times\mathbb{E}_1^\theta.$$
In particular, this yields
$$u_*\in W_p^{3}((\tau,T);L_q(\Omega))\cap W_p^{2}((\tau,T);W_q^2(\Omega)),$$
$$\theta_*\in W_r^{2}((\tau,T);L_s(\Omega))\cap W_r^{1}((\tau,T);W_s^2(\Omega)),$$
for each $\tau\in (0,T)$, as $\varepsilon\in (0,1)$ was arbitrary.

Moreover, if all nonlinearities $c,b,k$ and $Q$ are $C^m$-mappings, where $m\in\mathbb{N}$, then also $\phi\in C^m((1-r,1+r);\mathbb{E}_1^u\times\mathbb{E}_1^\theta)$ by the implicit function theorem. Inductively, this yields
$$[t\mapsto t^m\partial_t^{m}(u_*(t),\theta_*(t))]\in \mathbb{E}_1^u\times\mathbb{E}_1^\theta$$
and therefore
$$u_*\in W_p^{m+2}((\tau,T);L_q(\Omega))\cap W_p^{m+1}((\tau,T);W_q^2(\Omega)),$$
$$\theta_*\in W_r^{m+1}((\tau,T);L_s(\Omega))\cap W_r^{m}((\tau,T);W_s^2(\Omega)).$$
We have thus proven the following result.
\begin{theorem}\label{thm:HR}
Let the conditions of Theorem \ref{thm:main} be satisfied. Then
the unique solution
$$u\in W_p^2((0,T);L_q(\Omega))\cap W_p^1((0,T);W_q^2(\Omega))$$
$$\theta\in W_r^1((0,T);L_s(\Omega))\cap L_r((0,T);W_s^2(\Omega))$$
of \eqref{eq:WestPenn} with $g_j=h_\ell=0$ satisfies
$$u\in W_p^{3}((\tau,T);L_q(\Omega))\cap W_p^{2}((\tau,T);W_q^2(\Omega)),$$
$$\theta\in W_r^{2}((\tau,T);L_s(\Omega))\cap W_r^{1}((\tau,T);W_s^2(\Omega)),$$
for each $\tau\in (0,T)$.

If, in addition, $c,b,k$ and $Q$ are $C^m$-mappings, it holds that
$$u\in W_p^{m+2}((\tau,T);L_q(\Omega))\cap W_p^{m+1}((\tau,T);W_q^2(\Omega)),$$
$$\theta\in W_r^{m+1}((\tau,T);L_s(\Omega))\cap W_r^{m}((\tau,T);W_s^2(\Omega)).$$
for each $\tau\in (0,T)$.
\end{theorem}
\begin{remark}
Under the conditions of Theorem \ref{thm:HR} one can also prove \emph{joint} \emph{time-space} regularity by an application of the parameter trick in \cite[Section 9.4]{PS16}. We refrain from giving the details.
\end{remark}

\section{Equilibria and Long-Time Behaviour}\label{sec:Equil}

The equilibria $(u_*,\theta_*)$ of \eqref{eq:WestPenn} with $g_j=0$ and $h_\ell=(1-\ell)\theta_a$ are determined by the equations
\begin{equation}
\label{eq:EQWestPenn}
\begin{aligned}
-c^2(\theta)\Delta u_* &= 0,&&\text{in }\Omega,\\
-\kappa_a\Delta\theta_*+\rho_bC_bW(\theta_*-\theta_a)&=Q(0),&&\text{in }\Omega,\\
     \mathcal{B}_ju_*&=0 ,&&\text{on }\partial\Omega,\\
     \mathcal{B}_\ell \theta_*&=(1-\ell)\theta_a ,&&\text{on }\partial\Omega.
\end{aligned}
\end{equation}
Let us assume that $c^2(\tau)\ge c_0>0$ for all $\tau\in\mathbb{R}$. It follows that $u_*=0$ if $j=0$ or $u_*$ is an arbitrary constant if $j=1$ .

Concerning $\theta$, we observe that if $Q(0)=0$, then $\theta_*=\theta_a$ is the unique solution of $\eqref{eq:EQWestPenn}_{2,4}$.
We will show that in case $j=0$, the equilibrium $(u_*,\theta_*)=(0,\theta_a)$ is exponentially stable (in the sense of Lyapunov). In a first step, we define $\tilde{\theta}:=\theta-\theta_a$, so that we may consider the problem
\begin{equation}
\label{eq:WestPennEquilibria1}
\begin{aligned}
u_{tt} -\tilde{c}^2(\tilde{\theta})\Delta u-\tilde{b}(\tilde{\theta})\Delta u_t &= \tilde{k}(\tilde{\theta})(u^2)_{tt},&&\text{in }(0,T)\times\Omega,\\
\rho_aC_a\tilde{\theta}_t-\kappa_a\Delta\tilde{\theta}+\rho_bC_bW\tilde{\theta}&=Q(u_t),&&\text{in }(0,T)\times\Omega,\\
     u&= 0,&&\text{in }(0,T)\times\partial\Omega,\\
     \mathcal{B}_\ell\tilde{\theta}&= 0,&&\text{in }(0,T)\times\partial\Omega,\\
      (u(0),u_t(0)) &= (u_0,u_1) ,&&\text{in }\Omega,\\
      \tilde{\theta}(0)&= \tilde{\theta}_0 ,&&\text{in }\Omega,
\end{aligned}
\end{equation}
where $\tilde{\theta}_0:= \theta_0-\theta_a$ and $\tilde{f}(\tau):= f(\tau+\theta_a)$ for $f\in\{c,b,k\}$. Observe that 
$$\theta_0\in B_{sr}^{2-2/r}(\Omega)\quad\Longleftrightarrow\quad \tilde{\theta}_0\in B_{sr}^{2-2/r}(\Omega)$$
as $\theta_a$ is constant and $\Omega$ is bounded.

We define the function spaces
$$\mathbb{E}_0^u(\mathbb{R}_+):=L_p(\mathbb{R}_+;L_q(\Omega)),\quad \mathbb{E}_0^{\tilde{\theta}}(\mathbb{R}_+):=L_r(\mathbb{R}_+;L_s(\Omega)),$$
$$\mathbb{E}_1^u(\mathbb{R}_+):=\{u\in W_p^2(\mathbb{R}_+;L_q(\Omega))\cap W_p^1(\mathbb{R}_+;W_q^2(\Omega)):u=0\ \text{on}\ \partial\Omega\},$$
$$\mathbb{E}_1^{\tilde{\theta}}(\mathbb{R}_+):=\{\tilde{\theta}\in W_r^1(\mathbb{R}_+;L_s(\Omega))\cap L_r(\mathbb{R}_+;W_s^2(\Omega)):\mathcal{B}_\ell\tilde{\theta}=0\ \text{on}\ \partial\Omega\},$$
$$\mathbb{X}_\gamma^u:=\{(u_0,u_1)\in W_q^2(\Omega)\times B_{qp}^{2-2/p}(\Omega):u_1|_{\partial\Omega}=0\ \text{if}\ 1-1/2q>1/p,\ u_0|_{\partial\Omega}=0\} ,$$
and
$$\mathbb{X}_\gamma^{\tilde{\theta}}:=\{\tilde{\theta}_0\in B_{sr}^{2-2/r}(\Omega):\mathcal{B}_\ell\tilde{\theta}_0=0\ \text{on}\ \partial\Omega\ \text{if}\ 1/2-1/2s>1/r\}.$$
For $\mathbb{F}\in \{\mathbb{E}_0^u,\mathbb{E}_1^u,\mathbb{E}_0^{\tilde{\theta}},\mathbb{E}_1^{\tilde{\theta}}\}$ we define furthermore
$$v\in e^{-\omega}\mathbb{F}(\mathbb{R}_+) \quad:\Longleftrightarrow\quad [t\mapsto e^{\omega t}v(t)]\in \mathbb{F}(\mathbb{R}_+),\quad \omega\ge 0,$$
and a mapping
$$\Phi:e^{-\omega}\mathbb{E}_1^u(\mathbb{R}_+)\times e^{-\omega}\mathbb{E}_1^{\tilde{\theta}}(\mathbb{R}_+)\times \mathbb{X}_\gamma^u\times \mathbb{X}_\gamma^{\tilde{\theta}}\to e^{-\omega}\mathbb{E}_0^u(\mathbb{R}_+)\times e^{-\omega}\mathbb{E}_0^{\tilde{\theta}}(\mathbb{R}_+)\times \mathbb{X}_\gamma^u\times \mathbb{X}_\gamma^{\tilde{\theta}}$$
by
$$\Phi(u,\tilde{\theta},u_0,u_1,\tilde{\theta}_0)=
	\begin{pmatrix}
	u_{tt} -\tilde{c}^2(\tilde{\theta})\Delta u-\tilde{b}(\tilde{\theta})\Delta u_t - \tilde{k}(\tilde{\theta})(u^2)_{tt}\\
	\rho_aC_a\tilde{\theta}_t-\kappa_a\Delta\tilde{\theta}+\rho_bC_bW\tilde{\theta}-Q(u_t)\\
	u(0)-u_0\\
	u_t(0)-u_1\\
	\tilde{\theta}(0)-\tilde{\theta}_0
	\end{pmatrix}.
$$
Note that the mapping $\Phi$ is well defined and
$$\Phi\in C^1\left(\mathbb{E}_1^u(\mathbb{R}_+)\times \mathbb{E}_1^{\tilde{\theta}}(\mathbb{R}_+)\times \mathbb{X}_\gamma^u\times \mathbb{X}_\gamma^{\tilde{\theta}};\mathbb{E}_0^u(\mathbb{R}_+)\times \mathbb{E}_0^{\tilde{\theta}}(\mathbb{R}_+)\times \mathbb{X}_\gamma^u\times \mathbb{X}_\gamma^{\tilde{\theta}}\right)$$
provided that
$$Q\in C^1\left(e^{-\omega}\dot{\mathbb{E}}_1^u(\mathbb{R}_+);e^{-\omega}\mathbb{E}_0^{\tilde{\theta}}(\mathbb{R}_+)\right),$$
where
$$\dot{\mathbb{E}}_1^u(\mathbb{R}_+):=W_p^1(\mathbb{R}_+;L_q(\Omega))\cap L_p(\mathbb{R}_+;W_q^2(\Omega)).$$
Moreover, $\Phi(0,0,0,0,0)=0$ and
$$D_{(u,\tilde{\theta})}\Phi(0,0,0,0,0)(\hat{u},\hat{\theta})=
\begin{pmatrix}
	\hat{u}_{tt} -\tilde{c}^2(0)\Delta \hat{u}-\tilde{b}(0)\Delta \hat{u}_t\\
	\rho_aC_a\hat{\theta}_t-\kappa_a\Delta\hat{\theta}+\rho_bC_bW\hat{\theta}-Q'(0)\hat{u}_t\\
	\hat{u}(0)\\
	\hat{u}_t(0)\\
	\hat{\theta}(0)
	\end{pmatrix}.$$
Let us recall that the Dirichlet- as well as the Neumann-Laplacian $\Delta_m$, $m\in \{D,N\}$ has the property of $L_r$-maximal regularity in $L_s(\Omega)$, see e.g. \cite[Section 6]{PS16}. Since for any $\alpha>0$, the spectral bound of the operator $(\Delta_m-\alpha I)$ in $L_s(\Omega)$ is strictly negative, it generates an exponentially stable analytic semigroup in $L_s(\Omega)$ with $L_r$-maximal regularity.

We note furthermore, that $\tilde{c}(0)=c(\theta_a)$ and $\tilde{b}(0)=b(\theta_a)$ are positive constants. Hence, \cite[Theorem 2.5]{MeWi11} in combination with the exponential stability of the semigroup, generated by $(\Delta_m-\alpha I)$ in $L_s(\Omega)$, implies that there is some $\omega_0>0$ such that for all $\omega\in [0,\omega_0)$, the operator
$$D_{(u,\tilde{\theta})}\Phi(0,0,0,0,0):e^{-\omega}\mathbb{E}_1^u(\mathbb{R}_+)\times e^{-\omega}\mathbb{E}_1^{\tilde{\theta}}(\mathbb{R}_+)\to e^{-\omega}\mathbb{E}_0^u(\mathbb{R}_+)\times e^{-\omega}\mathbb{E}_0^{\tilde{\theta}}(\mathbb{R}_+)\times \mathbb{X}_\gamma^u\times \mathbb{X}_\gamma^{\tilde{\theta}}$$
is invertible. By the implicit function theorem, there exists some $\delta>0$ and a mapping
$$\psi\in C^1\left(\mathbb{B}_{\mathbb{X}_\gamma^u\times \mathbb{X}_\gamma^{\tilde{\theta}}}((0,0,0),\delta);e^{-\omega}\mathbb{E}_1^u(\mathbb{R}_+)\times e^{-\omega}\mathbb{E}_1^{\tilde{\theta}}(\mathbb{R}_+)\right)$$
such that $\psi(0,0,0)=(0,0)$ and
$$\Phi(\psi(u_0,u_1,\tilde{\theta}_0),(u_0,u_1,\tilde{\theta}_0))=0$$
for all $(u_0,u_1,\tilde{\theta}_0)\in \mathbb{B}_{\mathbb{X}_\gamma^u\times \mathbb{X}_\gamma^{\tilde{\theta}}}((0,0,0),\delta)$. Since $\psi(0,0,0)=0$ and $\psi$ is continuously differentiable, it follows that for each $r\in (0,\delta)$, there exists a constant $C=C(r)>0$ such that
$$\|\psi(u_0,u_1,\tilde{\theta}_0)\|_{e^{-\omega}\mathbb{E}_1^u(\mathbb{R}_+)\times e^{-\omega}\mathbb{E}_1^{\tilde{\theta}}(\mathbb{R}_+)}\le C\|(u_0,u_1,\tilde{\theta}_0)\|_{\mathbb{X}_\gamma^u\times \mathbb{X}_\gamma^{\tilde{\theta}}}$$
holds for all $(u_0,u_1,\tilde{\theta}_0)\in \mathbb{B}_{\mathbb{X}_\gamma^u\times \mathbb{X}_\gamma^{\tilde{\theta}}}((0,0,0),r)$.

For the solution $(u,\tilde{\theta})=\psi(u_0,u_1,\tilde{\theta}_0)$ of \eqref{eq:WestPennEquilibria1}, this implies the estimate
\begin{multline}\label{eq:expstab}
e^{\omega t}\left(\|u(t)\|_{W_q^2(\Omega)}+\|u_t(t)\|_{B_{qp}^{2-2/p}(\Omega)}+\|\tilde{\theta}(t)\|_{B_{sr}^{2-2/r}(\Omega)}\right)\le\\
\le C\left(\|u_0\|_{W_q^2(\Omega)}+\|u_1\|_{B_{qp}^{2-2/p}(\Omega)}+\|\tilde{\theta}_0\|_{B_{sr}^{2-2/r}(\Omega)}\right)
\end{multline}
for all $t\ge 0$. We summarize these considerations in
\begin{theorem}\label{thm:Equil}
Let $\Omega\subset\mathbb{R}^d$ be a bounded domain with boundary $\partial\Omega\in C^2$ and suppose that $c,b,k\in C^1(\mathbb{R})$ with $b(\tau)\ge b_0>0$ and $c^2(\tau)\ge c_0>0$ for all $\tau\in\mathbb{R}$. Assume furthermore that $p,q,r,s\in (1,\infty)$ such that
$$\frac{d}{q}<2,\quad \frac{2}{r}+\frac{d}{s}<2$$
and
$$Q\in C^1\left(e^{-\omega}(W_p^1(\mathbb{R}_+;L_q(\Omega))\cap L_p(\mathbb{R}_+;W_q^2(\Omega)));e^{-\omega}L_r(\mathbb{R}_+;L_s(\Omega))\right),$$
with $Q(0)=0$. Assume that $1-1/2q\neq 1/p$ and $1-\ell/2-1/2s\neq 1/r$.

Then there are $\delta>0$ and $\omega_0>0$ such that for all $\omega\in[0,\omega_0)$,
$$u_0\in W_q^2(\Omega),\quad u_1\in B_{qp}^{2-2/p}(\Omega),\quad {\theta}_0\in B_{sr}^{2-2/r}(\Omega),$$
with
\begin{itemize}
\item $u_0|_{\partial\Omega}=0$,
\item $u_1|_{\partial\Omega}=0$ if $1-1/2q>1/p$,
\item $\mathcal{B}_\ell{\theta}_0=(1-\ell)\theta_a$ on $\partial\Omega$ if $1-\ell/2-1/2s>1/r$
\end{itemize}
and
$$\|u_0\|_{W_q^2(\Omega)}+\|u_1\|_{B_{qp}^{2-2/p}(\Omega)}+\|{\theta}_0-\theta_a\|_{B_{sr}^{2-2/r}(\Omega)}\le\delta,$$
there exists a unique global solution $(u,\theta)$ of \eqref{eq:WestPenn} with
$$u\in e^{-\omega}(W_p^2(\mathbb{R}_+;L_q(\Omega))\cap W_p^1(\mathbb{R}_+;W_q^2(\Omega)))$$
$${\theta}-\theta_a\in e^{-\omega}(W_r^1(\mathbb{R}_+;L_s(\Omega))\cap L_r(\mathbb{R}_+;W_s^2(\Omega))).$$
Moreover, there exists a constant $C>0$ such that the estimate
\begin{multline*}
\|u(t)\|_{W_q^2(\Omega)}+\|u_t(t)\|_{B_{qp}^{2-2/p}(\Omega)}+\|{\theta}(t)-\theta_a\|_{B_{sr}^{2-2/r}(\Omega)}\le\\
\le Ce^{-\omega t}\left(\|u_0\|_{W_q^2(\Omega)}+\|u_1\|_{B_{qp}^{2-2/p}(\Omega)}+\|{\theta}_0-\theta_a\|_{B_{sr}^{2-2/r}(\Omega)}\right)
\end{multline*}
holds for all $t\ge 0$.
\end{theorem}

\begin{remark}
In \cite{NikSaid22}, the authors proved Theorem \ref{thm:Equil} for the case $p=q=s=2$, $d\in \{2,3\}$ under more restrictive assumptions on the initial data $(u_0,u_1,\theta_0)$ as well as on the nonlinearities $c,k,Q$ by means of higher order energy methods/estimates. Furthermore, in \cite{NikSaid22} it is assumed that the function $b$ is constant. Thus, Theorem \ref{thm:Equil} may be understood as a generalization of the results in \cite{NikSaid22}.
\end{remark}
\begin{remark}
In case $j=1$ (Neumann boundary conditions for $u$), one has to deal with a family of equilibria $(u_*,\theta_*)$, where $u_*=\mathsf{r}\in\mathbb{R}$ is constant and $\theta_*=\theta_a$. In this case, one can use the same strategy as in \cite{SiWi17} to show that each equilibrium $(\mathsf{r},\theta_*)$, with $\mathsf{r}\in\mathbb{R}$ being close to zero, is \emph{normally stable}. We refrain from giving the details and refer the interested reader to \cite{PSZ09} and \cite{SiWi17}.
\end{remark}

\bibliographystyle{abbrv}
\bibliography{WesterveltPennes}

\end{document}